\documentclass[letterpaper, 10 pt, conference]{ieeeconf}  

\usepackage{graphicx,psfrag,epsfig,epsf}
\usepackage{subfigure}
\usepackage{amsmath}
\usepackage{amsfonts}
\usepackage{amssymb}
\usepackage{verbatim}
\usepackage{latexsym}
\usepackage{color}
\usepackage{algorithm2e}
\usepackage{tikz}
\usetikzlibrary{automata, calc}
\usetikzlibrary{shapes}
\usepackage{comment}
\usepackage{mathtools}

\usepackage{verbatim}

\usepackage{latexsym}

\newcommand{\Post}{\operatorname{Post}}

\renewcommand{\rm}{\textcolor{red}}

\newcommand{\Ac}{\mathrm{Ac}}
\newcommand{\CoAc}{\mathrm{Coac}}
\newcommand{\Trim}{\mathrm{Trim}}

\newtheorem{proposition}{Proposition}
\newtheorem{assumption}{Assumption}
\newtheorem{theorem}{Theorem}
\newtheorem{corollary}{Corollary}

\newtheorem{definition}{Definition}
\newtheorem{problem}{Problem}
\newtheorem{remark}{Remark}

\begin{document}

	\IEEEoverridecommandlockouts                              
	
	
	
	\overrideIEEEmargins
	
	
	
	
	
	
	
	
	

 \title{\LARGE \bf 
On Incremental Design of Large-Scale Networks of \\
Nondeterministic Metric Finite State Systems 
        }

	\author{Giordano Pola
		\thanks{G. Pola is with the Department of Information Engineering, Computer Science and Mathematics, Center of Excellence DEWS, University of L{'}Aquila, Via Vetoio, 67100 L{'}Aquila, Italy, {\tt\small giordano.pola@univaq.it}.}
   	}
	
	\maketitle
	
	\thispagestyle{empty}
	
	\pagestyle{empty}

	
	\begin{abstract}
In this paper, we address control design of Networks of nondeterministic and metric finite state Systems (NoS) where some systems are plants, others act as controllers, and some are uncontrolled processes. 
The control architecture of NoS extends decentralized control architectures.
The problem addressed consists in designing local controllers for enforcing on the corresponding plants, local specifications  given by regular languages over plants' state alphabet, up to desired accuracies. The approach used to solve the problem is termed \textit{incremental} and consists in finding solutions on an increasing sequence of NoS extracted from the original NoS.
    \end{abstract}
	
	\IEEEoverridecommandlockouts
	
	\thispagestyle{empty}
	
	\pagestyle{empty}
	
	\vspace{5mm}
	\textbf{keywords: } Networks of systems, decentralized control, incremental design.

\section{Introduction} \label{sec1}
In this paper we consider a Network of Systems (NoS) that is given by the interconnection of nondeterministic and metric finite state systems, modeling plants to-be-controlled, controllers to-be-designed and other processes that does not need to be controlled. 
NoS architecture is useful for describing or approximating possibly large-scale interconnected systems featuring continuous and discrete dynamics, and 
are relevant in modeling diverse applications of interest as micro-grids, biological systems, etc. 
The control problem we consider consists in designing local controllers for forcing the corresponding plants to satisfy each, a local specification expressed as a regular language defined over alphabet of plant states, see e.g. \cite{ARC}. Since NoS are possibly large-scale, appropriate methodologies are required rather than centralized architectures to tame complexity inherent to control design. Among them, we recall decentralized control architecture (DecCA) and distributed control architecture (DisCA).  
In DecCA, each plant $P_i$ is assigned to a controller $C_i$ that provides inputs to $P_i$ using information of the state/output of $P_i$. 
In DisCA, each controller $C_i$ gets information of the state/output of $P_i$ and of other controllers $C_j$, and sends inputs to $P_i$. 
In NoS, each controller $C_i$ gets output information from some plants $P_j$, not necessarily including $P_i$, and sends inputs to 
$P_i$. In this way, NoS can be viewed as a generalization of DecCA. Comparison among the three architectures follows.  
From the information pattern point of view, controllers in DisCA and NoS have more information than those in DecCA. DisCA and NoS are not comparable in this regard. 
As a consequence, controllers in DisCA and NoS are able to enforce, generally speaking, larger parts of the specifications than those in DecCA. 
On the other hand, from the Time Computational Complexity (TCC) point of view, DisCA and NoS exhibit higher order TCC than DecCA. TCC of DisCA and NoS is in general comparable because controllers in DecCA embed inside information about the plants they control. 
Another direction for comparison is on whether some communication infrastructures among controllers may be designed or not: some physical constraints may prevent this from happening, and in case it is possible, it may be expensive and not reasonable to the designer. As a consequence, DisCA and NoS exhibit lower expensive communication infrastructure than  DicCA. Finally, while NoS and DecCA may allow cooperative but also non-cooperative control, see e.g. \cite{PolaAUT26} for DecCA, DisCA is targeted to consider only cooperative control. 
Since NoS is much closer to DecCA than DisCA, we provide hereafter a brief overview of existing related results. Decentralized supervisory control of DES, see e.g. \cite{Rudie92}, has been extensively studied, which however, considers specifications given in terms of regular languages of events while here we consider regular languages of states. Hence, DES framework cannot be directly used for the control design of NoS. Work related to the present paper is reported in \cite{DallalCDC15,KimCDC15,Meyer15,PolaTAC17}, which use discrete abstractions to design decentralized controllers for enforcing persistency, safety, Linear Temporal Logic and regular language specifications, respectively. Papers closer to the framework of NoS are \cite{PolaAUT26} and \cite{PolaIEEECSL24} that address decentralized control of networks of finite state systems with regular language-based specification and assuming exact knowledge of plant states, whereas in NoS information about plant states is given through their outputs. In this paper  we propose an approach to the design of NoS that we term incremental and consists in finding solutions on an increasing sequence of NoS extracted from the original NoS. As steps in the procedure increase, the controllers designed may enforce larger parts of the specifications than those obtained in the previous steps. This procedure terminates when: (i) specifications are fully enforced, or (ii) the obtained sub-NoS coincides with the original NoS, or (iii) there are no resources available to support computations needed to solve the control problem. Our approach may allow providing solutions to the control problem even in  case (iii) and allows the use of parallel computing architectures which is important from the scalability point of view. \\
This paper is organized as follows. Section \ref{sec2} introduces  notation and preliminary definitions. 
Section \ref{sec3} introduces NoS with one plant and one controller and control problem formulation. Section \ref{sec4} solves the control problem for a special NoS, Section \ref{sec5} discusses the incremental approach, Section \ref{sec6} extends preceding results to NoS with multiple plants, controllers and specifications,  Section \ref{sec7} discusses TCC analysis, Section \ref{sec8} presents an example, and Section \ref{sec9} gives concluding remarks.

\section{Notation and preliminary definitions} \label{sec2}
\subsection{Notation}
Symbol $\wedge$ denotes the logical conjunction and symbol $\varnothing$, the empty set. Given a finite set $X$, $\vert X \vert$ denotes the cardinality of $X$, and $2^X$  the power set of $X$, that is the collection of all subsets of $X$. 
Given two sets $X$ and $Y$, the symbol $X\,\backslash \,Y$ denotes the difference set, i.e. $X\,\backslash \,Y=\{x\in X\,|\,x \notin Y\}$. 
Function $f:X\rightarrow Y$ is empty if $X=\varnothing$, and its inverse map is $f^{-1}:Y\rightarrow 2^X$ defined by $f^{-1}(y)=
\{x\in X \,|\, f(x)=y\}$, $y\in Y$. 
Symbols $\mathbb{N}$, $\mathbb{R}$, $\mathbb{R}^{+}$ and $\mathbb{R}_{0}^{+}$ denote the set of non-negative integer, real, positive real, and non-negative real numbers, respectively. Given $a,b\in\mathbb{N}$ we set $[a;b]=\{x\in\mathbb{N}\,|\, a\leq x \leq b\}$.
\begin{definition}
A directed graph $\mathcal{D}$ is a pair $(\mathcal{V},\mathcal{E})$ where $\mathcal{V}$ is the finite set of vertices and $\mathcal{E}\subseteq \mathcal{V} \times \mathcal{V}$ is the set of directed edges. 
A directed path of $\mathcal{D}$ with length $l \in \mathbb{N}$ is a sequence $v_1\,v_2\,...\,v_l$ of vertices $v_i\in \mathcal{V}$, $i\in [1;N]$, satisfying $(v_i,v_{i+1})\in \mathcal{E}$ for $i \in [1;l-1]$.     
\end{definition}

\subsection{Transition systems and 
regular languages}\label{sec:ApproxEquiv}
We recall from e.g. \cite{Cassandras} some notions on language theory. Let $Y$ be a finite set representing the alphabet. A word over $Y$ is a finite sequence 
of symbols in $Y$; its length is the number of its entries. 
The empty word is denoted by $\varepsilon$. 
The symbol $Y^\ast$ denotes the Kleene closure of $Y$, that is the collection of all words over $Y$ including $\varepsilon$. 
A language $L$ over $Y$ is a subset of $Y^\ast$. 
We now recall the following
\begin{definition}
\label{systems}
A transition system is a tuple 
\begin{equation}
\label{TransSys}
 T=(X,X_0,U,\xrightarrow[\text{ }\text{ }\text{ }\text{ }]{\text{ }},X_m,Y,H),   
\end{equation}
 consisting of a set of states $X$, a set of initial states $X_0 \subseteq X$, a set of inputs $U$, a transition relation $\xrightarrow[\text{ }\text{ }\text{ }\text{ }]{\text{ }} \subseteq X\times U\times X$, a set of marked states $X_m \subseteq X$, a set of outputs $Y$ and an output function $H:X\rightarrow Y$.
\end{definition}
A transition $(x,u,x^{\prime})\in
\xrightarrow[\text{ }\text{ }\text{ }\text{ }]{\text{ }}$ of $T$ is denoted by $x
\xrightarrow[\text{ }\text{ }\text{ }\text{ }]{u}x^{\prime}$.  
The evolution of transition systems is captured by the notions of state, input and output runs. Given a sequence of transitions of $T$
\begin{equation}
\label{seqtrans}
x(0) 
\xrightarrow[\text{ }\text{ }\text{ }\text{ }]{u(0)}x(1)
\xrightarrow[\text{ }\text{ }\text{ }\text{ }]{u(1)} \,{...}\,x(l-1) 
\xrightarrow[\text{ }\text{ }\text{ }\text{ }]{u(l-1)}
\end{equation}
with $x(0) \in X_0$, the sequences 
$r_X: \, x(0) \, x(1) \, ... \, x(l)$, 
$r_U: \, u(0) \, u(1) \, ... \, u(l-1)$ and $r_Y: H(x(0)) \, H(x(1)) \, ... \, H(x(l))$  
are called a \textit{state run}, an \textit{input run} and an \textit{output run} of $T$, respectively.
Length of $r_X$ and $r_Y$ is $l+1$ while length of $r_U$ is $l$. 
Transition system $T$ in (\ref{TransSys}) is \textit{empty} if $X_0=\varnothing$, \textit{finite} if $X$ and $U$ are finite sets, \textit{metric} if $X$ is equipped with a metric $\mathbf{d}:X\times X\rightarrow\mathbb{R}_{0}^{+}$, \textit{deterministic} if for any $x\in X$ and $u\in U$ there exists at most one transition $x 
\xrightarrow[\text{ }\text{ }\text{ }\text{ }]{u} x^+$ and \textit{nondeterministic}, otherwise. 
\begin{definition}
   The \textit{language} of $T$, denoted $\mathcal{L}(T)$, is the collection of all its output runs. 
The \textit{marked language} of $T$, denoted $\mathcal{L}_m(T)$, is the collection of all output runs $r_Y$ such that the corresponding transitions sequence in (\ref{seqtrans}) ends in a state $x_l\in X_m$. 
\end{definition}
\begin{definition}
 A language $L$ over a finite set $Y$ is \textit{regular} if there exists a finite transition system $T$ with input set $Y$ such that $L=\mathcal{L}_m (T)$.   
\end{definition}
We now recall some unary operations on transition systems, see e.g. \cite{Cassandras}. 
A transition system $T'=(X',X'_0,U',
\xrightarrow[\text{ }\text{ }'\text{ }\text{ }]{\text{}},X'_m,Y',H')$ is a \textit{sub-transition system} of $T=(X,X_0,U,\xrightarrow[\text{ }\text{ }\text{ }\text{ }]{\text{}},X_m,Y,H)$, denoted $T' \sqsubseteq T$, if $X'\subseteq X$, $X'_0\subseteq X_0$, $U'\subseteq U$, $
\xrightarrow[\text{ }\text{ }'\text{ }\text{ }]{\text{}} \subseteq \xrightarrow[\text{ }\text{ }\text{ }\text{ }]{\text{}}$, $X'_m \subseteq X_m$, $Y' \subseteq Y$ and $H'(x)=H(x)$ for all $x\in X'$. 
\begin{definition}
The accessible part of $T$ is the set of all states that are reachable from an initial state by following some sequence of input
symbols.
\end{definition} 
\begin{definition}
Given $T$, $u\in U$ and $x'\in X$ define:
\[
\begin{array}{l}
\Post_u(x')=\{x\in X\,|\, x'
\xrightarrow[\text{ }\text{ }\text{ }\text{ }]{u}
 x\};\\
\Post^{-1}_u(x')=\{x\in X\,|\, x
\xrightarrow[\text{ }\text{ }\text{ }\text{ }]{u}
x'\}.
\end{array}
\]
Given $T$, $u\in U$ and $X'\subseteq X$ define:
\[
\begin{array}{l}
 \Post_u(X')=\{x\in X\,|\,\exists x'\in X' \text{ s.t. } x'\xrightarrow[\text{ }\text{ }\text{ }\text{ }]{u} x\};\\
 \Post^{-1}_u(X')=\{x\in X\,|\,\exists x'\in X' \text{ s.t. } x
 \xrightarrow[\text{ }\text{ }\text{ }\text{ }]{u}
  x'\}.     
\end{array}
\]    
\end{definition}
Define the following recursive equations:
	\[
	\begin{array}
		{l}
		X(0):=X_m;\\
		X(k+1):=X(k) \bigcup_{u\in U}
		\left\{
		\begin{array}
			{l}
			x\in \Post^{-1}_u(X(k))\,|\\
			\,\Post_u(x)\subseteq X(k)
		\end{array}
		\right\}.
	\end{array}
	\]
 Sequence above admits a fixed point, i.e.  
 there exists $k'\in \mathbb{N}$ such that $X(k')=X(k'+1)$.
The co-accessible part $\CoAc(T)$ of $T$ is the sub-transition system of $T$ induced by $X(k')$, i.e. 
defined by the tuple 
\[
(X(k'),X_0 \cap X(k'),U,
\xrightarrow[\text{ }\text{ }'\text{ }\text{ }]{\text{}},X_m \cap X(k'),Y,H'),		
\]
 where $H'(x)=H(x)$ for all $x\in X'$, and
$x 
\xrightarrow[\text{ }\text{ }'\text{ }\text{ }]{u}
 x'$ if $x 
 \xrightarrow[\text{ }\text{ }\text{ }\text{ }]{u} x'$. When $T=\CoAc(T)$, it is said co-accessible. 
	By e.g. Example 2 of  \cite{CSL21Masciulli}, $\Ac$ and $\CoAc$ do not commute. 
	The trim of $T$, denoted $\Trim(T)$, is defined as $\Trim(T)=\Ac(\CoAc(T))$. 
	By definition, $\Trim(T)$, if not empty, is accessible and co-accessible. $T$ is trim if $T=\Trim(T)$.

\section{Networks of Systems and Control Problem} \label{sec3}
In this paper we consider a Network of Systems (NoS)  $\mathcal{N}=(\mathcal{V},\mathcal{E},h)$, 
where: $\mathcal{V}=\{\Sigma_1,\Sigma_2,...,\Sigma_N\}$ is the set of vertices, and $\Sigma_i$ are finite state systems; 
$\mathcal{E} \subseteq (\mathcal{V} \times \mathcal{V})\backslash \{(\Sigma_i,\Sigma_j)\in \mathcal{V} \times \mathcal{V}\,|\, i=j\}$ is the set of edges, describing interaction among systems; $h$ is the connection function, modeling information shared among systems. 
Given NoS $\mathcal{N}$, the pair $(\mathcal{V},\mathcal{E})$ can be viewed as a directed graph.
Finite state system $\Sigma_i$, $i\in [1;N]$, in $\mathcal{N}$, is described by the difference inclusion: 
\begin{equation}
\label{ith-plant}
\Sigma_i : 
\left\{
\begin{array}
{l}
x_i(t+1)\in f_i(x_i(t),u_i(t)),\\
x_i(0) \in X_{i,0},
x_i(t) \in X_i,
u_i(t) \in U_i, 
t\in\mathbb{N},
\end{array}
\right.
\end{equation}
where $x_i(t)$ and $u_i(t)$ are the state and the input at step $t\in\mathbb{N}$, respectively; $X_i$, 
$X_{i,0}\subseteq X_i$ and $U_i$ are the finite sets of states, initial states and inputs, respectively. Function $f_i : X_i \times U_i \rightarrow 2^{X_i}$ is the state transition function that is assumed to be total. A state trajectory of $\Sigma_i$, when considered as isolated from other systems $\Sigma_j$ in the NoS, is a finite sequence of states $x_i(0)\,x_i(1)\,...$ satisfying (\ref{ith-plant}) for some finite sequence of inputs $u_i(0)\,u_i(1)\,...$ . 
It is easy to see that systems $\Sigma_i$ are nondeterministic, in general. Interaction among systems $\Sigma_i$ in  $\mathcal{N}$ is formalized through connection function $h$ and inputs $u_i$. 
Function $h$ is specified for any $(\Sigma_j,\Sigma_i)\in \mathcal{E}$, 
by the collection of functions $h_{j,i} : X_j \rightarrow U_{j,i}$, for some finite sets $U_{j,i}$. Functions $h_{j,i}$ can be viewed as output functions associated with $\Sigma_j$. Input $u_i(.)$ of $\Sigma_i$, $i\in [1;N]$, is defined through connection function $h_{j,i}$,
for any $(\Sigma_j,\Sigma_i)\in \mathcal{E}$, 
by:
\begin{equation}
\label{ui}
\begin{array}
{l}
u_i(.)=(h_{1,i}(x_1(.)),...,h_{i-1,i}(x_{i-1}(.)),\\ \quad \quad \quad h_{i+1,i}(x_{i+1}(.)),..., h_{N,i}(x_N(.))),i \in [1;N], \\
U_i=U_{1,i} \times ... \times U_{i-1,i} \times U_{i+1,i} \times ... \times U_{N,i},
\end{array}
\end{equation}
for some collection of trajectories $x_i(.)$ of $\Sigma_i$, $i \in [1;N]$.
We denote by 
$\Sigma_{[1;N]}$, the finite state system resulting from the interconnection of $\Sigma_i$, which is defined by 
\begin{equation}
\label{sigmaenne}
\Sigma_{[1;N]} : 
\left\{
\begin{array}
{l}
\mathbf{x}(t+1)\in f(\mathbf{x}(t)),\\
\mathbf{x}(0) \in \mathbf{X}_0,
\mathbf{x}(t) \in \mathbf{X},
t\in\mathbb{N},\\
\end{array}
\right.
\end{equation}
where 
$\mathbf{x}(0):=(x_1(0),...,x_N(0))\in\mathbf{X}_0:=X_{1,0} \times ... \times X_{N,0}$, $\mathbf{x}(t):=(x_1(t),...,x_N(t))\in\mathbf{X}:=X_1 \times ... \times X_N$, 
$f:\mathbf{X}\rightarrow 2^{X_1} \times ... \times 2^{X_N}$, 
$f(\mathbf{x}(t)):=(f_1(x_1(t),u_1(t)),...,f_N(x_N(t),u_N(t)))$, 
and $u_i(t)$ are as in (\ref{ui}), $t \in \mathbb{N}$. A state trajectory $\mathbf{x}(.)$ of $\Sigma_{[1;N]}$ is a finite sequence of states $\mathbf{x}(0)\,\mathbf{x}(1)\,...\,\mathbf{x}(l)$ satisfying (\ref{sigmaenne}); length of $\mathbf{x}(.)$ is $l$. Among all systems in $\mathcal{V}$ we identify two systems: $\Sigma_1$ is the plant to-be-controlled, called $P$; $\Sigma_2$ is the controller to-be-designed for $P$, called $C$. 
From definition of $\Sigma_2$, controller $C$ is dynamic with output feedback from other systems $\Sigma_i$ in the NoS. 
At the beginning, $C$ and $h_{2,1}$ are unknown and are designed in the sequel to solve some control problems. We suppose that $X_1$ is metric with some metric $\mathbf{d}: X_1 \times X_1 \rightarrow \mathbb{R}^+_0 $. 
For later use, we give the following
\begin{definition}
\label{defsubNos}
A sub-NoS of NoS $\mathcal{N}=(\mathcal{V},\mathcal{E},h)$
is a NoS 
$\mathcal{N'}=(\mathcal{V}',\mathcal{E}',h')$ where 
$\mathcal{V}' \subseteq \mathcal{V}$, $\mathcal{E}'=\mathcal{E} \cap (\mathcal{V}' \times \mathcal{V}')$ and $h'$ is given by the collection of functions $h_{j,i}$ with $(\Sigma_j,\Sigma_i)\in \mathcal{E}'$.
\end{definition}
\begin{remark}
\label{RemSubNos} Given sub-NoS $\mathcal{N}'=(\mathcal{V}',\mathcal{E}',h')$ of $\mathcal{N}$, components of inputs $u_i(.)$ of $\Sigma_i$ with $\Sigma_i\in \mathcal{V}'$ can be split into internal and external to $\mathcal{N'}$. More precisely, given 
$u_i(.)=(u_{1,i}(.),...,u_{i-1,i}(.),u_{i+1,i}(.),..., u_{N,i}(.)),i \in [1;N]$, inputs $u_{j,i}(.)=h_{j,i}(x_j(.))$ with $\Sigma_j\in \mathcal{V}'$ are internal, while inputs $u_{j,i}:\mathbb{N}\rightarrow U_{j,i}$ with $\Sigma_j\in \mathcal{V}\,\backslash\,\mathcal{V}'$, are external. We assume that external inputs are not measurable.
\end{remark}

We can now introduce the control problem that we address: 

\begin{problem}
\label{problem1}
Consider a NoS $\mathcal{N}=(\mathcal{V},\mathcal{E},h)$, a specification described by a regular language $Q$ over the alphabet $X_1$ (set of states of $P$), and an accuracy $\theta\in\mathbb{R}^+_0$. Find a quadruplet
\begin{equation}
\label{quadruplet}
(C,X_{2,f},h_{2,1},\mathcal{R}_0)
\end{equation}
where $C$ is the controller, $X_{2,f}\subseteq X_2$, is a set of final states of $C$, $h_{2,1}: X_2 \rightarrow U_{2,1}$, is a connection function from $C$ to $P$, and $\mathcal{R}\subseteq X_{1,0} \times X_{2,0}$, is a relation  between initial states of $P$ and $C$,
such that 
for any state trajectory $\mathbf{x}(.)=(x_1(.),x_2(.),...,x_N(.))$ of $\Sigma_{[1;N]}$ with initial state 
\begin{equation}
\label{relinit}
\mathbf{x}(0)\in \mathcal{R}_0:=\mathcal{R} \times X_{3,0} \times ... \times X_{N,0},
\end{equation}
and final state of the controller $x_2(l)\in X_{2,f}$, where $l+1$ is the length of $\mathbf{x}(.)$, there exists a word $q_0\,q_1\,...\,q_l\in Q$ with length $l+1$, such that the following inequality holds:
\begin{equation}
    \label{ineq}
    \mathbf{d}(x_1(t),q_t)\leq \theta,\forall t\in[0;l].
\end{equation}
\end{problem}
For later purposes, we need to introduce the following
\begin{definition}
\label{defPartSpec}
    The part of the specification $Q$ implemented by $\Sigma_{[1;N]}$ with controller $C$, final states of controller in $X_{2,f}$, connection function $h_{2,1}$ and relation of initial states $\mathcal{R}_0$, denoted by $\mathcal{Q}(\Sigma_{[1;N]},C,X_{2,f},h_{2,1},\mathcal{R}_0)$,  is the collection of words $q_0\,q_1\,...\,q_l\in Q$ for which there exists a state trajectory $\mathbf{x}(.)=(x_1(.),x_2(.),...,x_N(.))$ of $\Sigma_{[1;N]}$ with initial state satisfying  (\ref{relinit}), length $l+1$, and final state of the controller $x_2(l)\in X_{2,f}$ such that (\ref{ineq}) holds. 
\end{definition}
In this paper we make the following:
\begin{assumption} 
\label{ass0}
Given NoS $\mathcal{N}$: pair $(C,\Sigma_i)\in\mathcal{E}$ if and only if $\Sigma_i=P$; for any $\Sigma_i\in \mathcal{V}$ there exists a directed path  from $\Sigma_i$ to either $C$ or $P$. 
\end{assumption}
Problem \ref{problem1} addresses control design of NoS $\mathcal{N}$ where only one plant, one controller and one specification are considered. The solution to this problem is given in Sections \ref{sec4} and \ref{sec5}. Extension to the case of NoS $\mathcal{N}$ with multiple plants, controllers and specifications is discussed in Section \ref{sec6}.
\section{Control design of a special NoS}\label{sec4}
In this section we start by deriving results for solving Problem \ref{problem1} for a special NoS described by $\mathcal{N}'=(\mathcal{V}',\mathcal{E}',h')$, 
where: $\mathcal{V}'=(\Sigma'_1,\Sigma'_2,\Sigma'_3)$; $\mathcal{E}'=
(\mathcal{V}'\times \mathcal{V}')\,
\backslash \{(\Sigma'_i,\Sigma'_j)\in \mathcal{V}' \times \mathcal{V}'\,|\, i=j\} \backslash \, \{(\Sigma'_2,\Sigma'_3)\}$; connection function $h'$ is  specified later on. Dynamics associated to $\Sigma'_j$, $j=1,2,3$, are of the form of $\Sigma_i$ in (\ref{ith-plant}) and specified by
\begin{equation}
\label{jth-plant}
\Sigma'_j : 
\left\{
\begin{array}
{l}
x'_j(t+1)\in f'_j(x'_j(t),u'_j(t)),\\
x'_j(0) \in X'_{j,0},
x'_j(t) \in X'_j,
u'_j(t) \in U'_j,t\in\mathbb{N},
\end{array}
\right.
\end{equation}
where $x'_j(t)$ and $u'_j(t)$ are the state and the input at step $t\in\mathbb{N}$, respectively; $X'_j$, $X'_{j,0}\subseteq X'_{j}$ and $U'_j$ are the finite sets of states, of initial states and of inputs, respectively. Function $f'_j : X'_j \times U'_j \rightarrow 2^{X'_j}$ is the state transition function and assumed to be total. Sets $U'_j$ are given by $U'_1=U'_{2,1} \times U'_{3,1} \times W_1$, 
$U'_2=U'_{1,2} \times U'_{3,2}$, 
$U'_3=U'_{1,3}\times W_3$, 
for some finite sets $U'_{i,j}$ and $W_i$, and connection function $h'$ is specified by $h'_{j,i} : X'_j \rightarrow U'_i$ for $(\Sigma'_j,\Sigma'_i)\in \mathcal{E}'$.
Inputs $u'_j(.)$ are given by 
\begin{equation}
\label{uprimi}
\begin{array}
{l}
u'_1(.)=(h'_{2,1}(x'_2(.)),h'_{3,1}(x'_3(.)),w_1(.)),\\
u'_2(.)=(h'_{1,2}(x'_1(.)),h'_{3,2}(x'_3(.)))\\
u'_3(.)=(h'_{1,3}(x'_1(.)),w_3(.)),\\
\end{array}
\end{equation}
where for $j=1,3$, inputs $w_j:\mathbb{N} \rightarrow W_j$ are external to $\Sigma'_i$.  
We denote by $\Sigma(\mathcal{N}')$ the system obtained by coupling systems $\Sigma'_j$ in (\ref{jth-plant}) through inputs $u'_j(.)$ in (\ref{uprimi}).
We interpret $\Sigma'_1$ as the plant, called $P'$, and $\Sigma'_2$ as the controller, called $C'$. We suppose that $X'_1$ is metric with metric $\mathbf{d}':X'_1 \times X'_1 \rightarrow \mathbb{R}^+_0$. At the beginning, controller $C'$ and connection function $h'_{2,1}(.)$ (from $C'$ to $P'$) are unknown and are designed later on to solve Problem \ref{problem1}.
By using the results established in \cite{PolaTAC17}, it is possible to associate a transition system $T_{Q'}$ to specification $Q'$, called dual transition system, specified by
\begin{equation}
\label{dualTS}
T_{Q'}=(X_{q},X_{q,0},U_{q},
\xrightarrow[\text{ }\text{ }q\text{ }\text{ }]{\text{}},X_{q,m},Y_q,H_{q})
\end{equation}
with $Y_q=X'_1$, and satisfying the following properties:
$T_{Q'}$ is finite, accessible, co-accessible and metric with metric $\mathbf{d}'$;
the marked output language $\mathcal{L}_m^y(T_{Q'})$ of $T_{Q'}$ coincides with $Q'$.
\incmargin{1em}
\restylealgo{boxed}\linesnumbered
\begin{algorithm*}[t]
\label{alg1}
\SetLine
\textbf{input}\;
NoS $\mathcal{N}'$; 
specification $Q'$;
accuracy $\theta'$\;
\textbf{init}\;
$X^c=X^c_0:=\{(x'_1,x'_3,x_q)\in X'_{1,0} \times X'_{3,0}  \times X_{q,0} \,|\, \mathbf{d}'(x'_1,H_q(x_q))\leq \theta'\}$; 
$U^c:=U'_2$; $
\xrightarrow[\text{ }\text{ }c\text{ }\text{ }]{}:=\varnothing$; $X^c_m:=\{(x'_1,x'_3,x_q)\in X^c \,|\, x_q\in X_{q,m}\}$; $Y^c:=Y_q$;
$H^c((x'_1,x'_3,x_q)):=H_q(x_q),$ $\forall (x'_1,x'_3,x_q) \in X^c$;
$T^c:=(X^c,X^c_0,U^c,
\xrightarrow[\text{ }\text{ }c\text{ }\text{ }]{\text{}},X^c_m,Y^c,H^c)$;
$h:X'_2\rightarrow U'_{2,1}$ empty function;
$\mathcal{R}':=\varnothing$;
$U'_{2,not}:=\varnothing$\; 
\ForEach {$u'_2=(u'_{1,2},u'_{3,2})\in U'_2$}{
\ForEach {$u'_{2,1}\in U'_{2,1}$ \text{ s.t. }
\begin{equation}
    \label{eqAlg1}
\begin{array}{c}
\forall (x'_1,x'_3,x_q) \in X^c \cap 
(((h'_{1,2})^{-1}( u'_{1,2})) 
\times ((h'_{3,2})^{-1}(u'_{3,2})) \times X_q ),
\forall (w_1,w_3) \in W_1 \times W_3,\\ 
\forall x^+_1\in f'_1(x'_1,(u'_{2,1},h'_{3,1}(x'_3)),w_1),
\forall x^+_3\in f'_3(x'_3,h'_{1,3}(x'_1),w_3),
\exists \,\, x_q 
\xrightarrow[\text{ }\text{ }q\text{ }\text{ }]{\text{}}
x_q^+ \text{ s.t. } \mathbf{d}'(x^+_1,H_q(x_q^+))\leq \theta'
 \end{array}  
\end{equation}
}
{
$X^c:=X^c \cup \{(x^+_1,x^+_3,x_q^+)\}$; 
$
\xrightarrow[\text{ }\text{ }c\text{ }\text{ }]{\text{}}
:=
\xrightarrow[\text{ }\text{ }c\text{ }\text{ }]{\text{}}
\cup \{((x'_1,x'_3,x_q),(h'_{1,2}(x^+_1),h'_{3,2}(x^+_3)),(x^+_1,x^+_3,x_q^+))\}$;
$H^c((x^+_1,x^+_3,x_q^+)):=H_q(x_q^+)$;
$h((x'_1,x'_3,x_q)):=u'_{2,1}$;
}
\If{ 
$\nexists u'_{2,1}\in U'_{2,1}$ \text{ s.t. (\ref{eqAlg1}) holds}
}
{$U'_{2,not}:=U'_{2,not} \cup \{u'_2\}$\;}
}
$T^c:=\Trim(T^c)$\;
\textbf{output} $(C',X'_{2,f},h'_{2,1},\mathcal{R}',T^c,U'_{2,not})$
 specified by \\
controller $C'$ defined by $X'_2:=X^c$; 
$X'_{2,0}:=X^c_0$; 
$f'_2((x'_1,x'_3,x_q),(u^+_{1,2},u^+_{3,2})):=\left\{
(x^+_1,x^+_3,x_q^+)\in X^c \,|\, (x'_1,x'_3,x_q) 
\xrightarrow[\text{ }\text{ }c\text{ }\text{ }]{(u^+_{1,2},u^+_{3,2})}
(x^+_1,x^+_3,x_q^+)
\right\}$\; 
$X'_{2,f}:=X^c_m$; 
$h'_{2,1}((x'_1,x'_3,x_q)):=h((x'_1,x'_3,x_q)), \forall (x'_1,x'_3,x_q)\in X'_2$\;
relation $\mathcal{R}'$ defined by $\mathcal{R}':=\{(x'_1,(x''_1,x''_3,x_q))\in X'_1\times X'_2\,|\,x'_1=x''_1 \}$.
\caption{Control design of NoS $\mathcal{N}'$.}
\end{algorithm*}
In the sequel, we denote a state of $X_q$ by $x_q$ and a transition of $T_{Q'}$ by $x_q 
\xrightarrow[\text{ }\text{ }q\text{ }\text{ }]{\text{}} x_q^+$, where label is skipped because it plays no role in the further developments. 
Solution to Problem \ref{problem1} is based on Algorithm \ref{alg1}. It is readily seen that 
Algorithm \ref{alg1} terminates in a finite number of steps. Moreover:
\begin{theorem}
\label{ThSimplified1}
Output of Algorithm \ref{alg1} solves Problem \ref{problem1}.
\end{theorem}

\begin{proof}
Consider any state trajectory $(x'_1(.),x'_2(.),x'_3(.))$ of $\Sigma(\mathcal{N}')$ with initial state $(x'_1(0),x'_2(0),x'_3(0))\in \mathcal{R}'$. \\
(Step $\#0$) By definition of $X'_2$ (lines 15, 13 and 4) and of $\mathcal{R}'$ (line 17), state 
$x'_2(0)=(x'_1,x'_3,x_q)$, with $x'_1=x'_1(0)$, $x'_3=x'_3(0)$ and $x_q\in X_q$, satisfies
\begin{equation}
\label{e0}
\mathbf{d}'(x'_1(0),H_q(x_q))=\mathbf{d}'(x'_1,H_q(x_q)) \leq \theta',
\end{equation}
which, by setting $q'_0=H_q(x_q)$, implies that the inequality in (\ref{ineq}) is satisfied for $t=0$. \\
(Step $\#1$) At step $t=0$, controller $C'$ does not know states $x'_1(0)$ and $x'_3(0)$ but only $u'_{1,2}(0)=h'_{1,2}(x'_1(0))$ and  
$u'_{3,2}(0)=h'_{3,2}(x'_3(0))$. 
In line 6, Algorithm \ref{alg1}  
selects $u'_{2,1} \in U'_{2,1}$ such that for any pair of external inputs $ (w_1,w_3) \in W_1\times W_3$ and for any state $x'_1(1)\in X'_1$ of $P'$ (reached by $x'_1(0)$ with $(u'_{2,1},h'_{3,1}(x'_3(0)),w_1)\in U'_{1}$) and for any state $x'_3(1)\in X'_3$ of $\Sigma'_3$ (reached by $x'_3(0)$ with $(h'_{1,3}(x'_1),w_3)\in U'_{3}$) there exists a transition $x_q 
\xrightarrow[\text{ }\text{ }q\text{ }\text{ }]{\text{}} x_q^+$ of the transition system $T_{Q'}$ such that
\begin{equation}
\label{e1}
\mathbf{d}'(x'_1(1),H_q(x_q^+))\leq \theta'.
\end{equation}
Since by line 16, $u'_{2,1} = h'_{2,1}(x'_1,x'_3,x_q)$, inclusions $x^+_1\in f'_1(x'_1,u'_{2,1},h'_{3,1}(x'_3),w_1)$ and $x^+_3\in f'_3(x'_3,h'_{1,3}(x'_1),w_3)$ in line 6 imply: 
\begin{equation}
\label{e2}
\begin{array}
{l}
x'_1(1)\in f'_1(x'_1(0), h'_{2,1}(x'_1(0),x'_3(0),x_q),h'_{3,1}(x'_3(0)),w_1(0)),\\
x'_3(1)\in f'_3(x'_3(0),h'_{1,3}(x'_1(0)),w_3(0)),
\end{array}
\end{equation}
where $(x'_1(1),x'_3(1))=(x^+_1,x^+_3)$, $(x'_1(0),x'_3(0))=(x'_1,x'_3)$ and $(w_1(0),w_3(0))=(w_1,w_3)$.
Moreover, by line 15, we obtain:
\begin{equation}
\label{e3}
x'_2(1)\in f'_2(x'_2(0),h'_{1,2}(x'_1(0)),h'_{3,2}(x'_3(0))).
\end{equation}
By combining (\ref{e2}) and (\ref{e3}) we get that $(x'_1(0),x'_2(0),x'_3(0))\,(x'_1(1),x'_2(1),x'_3(1))$ is a state trajectory of $\Sigma(\mathcal{N}')$. 
By setting $q'_1=H_q(x_q^+)$ and by (\ref{e1}), the inequality in (\ref{ineq}) is satisfied for $t=1$. By repeating (Step $\#1$) for $t=2$, we get that 
$(x'_1(0),x'_2(0),x'_3(0))\,(x'_1(1),x'_2(1),x'_2(1))\,(x'_1(2),x'_2(2),x'_3(2))$ is a state trajectory of $\Sigma(\mathcal{N}')$ for some states $x'_1(2)\in X'_1$, $x'_2(2)\in X'_2$ and $x'_3(2)\in X'_3$, and that the inequality in (\ref{ineq}) is satisfied at $t=2$, for some $q'_2$. 
Iteration above terminates when a marked state $x'_2(l')$, $l'\in\mathbb{N}$, is reached in $T^c$, or equivalently, when $x'_2(l')\in X'_{2,f}$. Such a state exists because $T^c$ is trim (line 13). Moreover, $x'_2(l')=(x'_1,x'_3,x_q)$, for some $x'_1$, $x'_3$ and $x_q$, where by lines 16 and 4, $x_q$ is marked also for $T_{Q'}$. As a consequence, the sequence $q'_0\, q'_1\, ...\,q'_{l'}$ is a word of $Q'$, which concludes the proof. 
\end{proof}
The result presented above can be used to solve Problem \ref{problem1} for NoS $\mathcal{N}$ by reformulating $\mathcal{N}$ as an appropriate NoS that is in the same form as NoS $\mathcal{N}'$. The main drawback of this approach is concerned with time computational complexity because the number of (needed-to-be-explored) transitions of nondeterministic system $\Sigma'_3$, composed of $N-2$ subsystems, scales doubly-exponential with $N-2$. 
In the following section we propose an approach to tame such a complexity.

\section{Incremental design of NoS}\label{sec5}
In this section we propose an approach to solve Problem \ref{problem1} that we call \textit{incremental}. 
Given the NoS $\mathcal{N}=(\mathcal{V},\mathcal{E},h)$, let $V$ be a set that is in bijection with $\mathcal{V}$ and such that $v_i\in V$ is uniquely associated to $\Sigma_i\in \mathcal{V}$; let $E\subseteq V\times V$ be the collection of pairs $(v_i,v_j)$ such that $(\Sigma_i,\Sigma_j)\in\mathcal{E}$. 
The pair $(V,E)$ is a directed graph.  
Given directed graph $D=(V,E)$, for any $k\in \mathbb{N}$ consider the sequence of directed graphs $D_k=(V_k,E_k)$ specified by: 
$V_0=\{v_1,v_2\}$, 
$E_0=\{(v_1,v_2),(v_2,v_1)\}\cap E$, 
$V_{k+1}=V_k \,\cup \, \{v_j\in V\backslash V_k \,|\, \exists v_i\in V_k \text{ s.t. } (v_j,v_i)\in E \}$, 
$E_{k+1}=E \cap ( V_{k+1}\times V_{k+1} )$, $k\in \mathbb{N}$. For any $k\in \mathbb{N}$ let $N_k=\vert V_k \vert$. 
In the sequel, we assume w.l.o.g. that for any $k\in \mathbb{N}$, $V_k=\{v_1,v_2,...,v_{N_k}\}$.
\begin{definition}
\label{defSigmaIK}

For any $k\in \mathbb{N}$ and any $i\in [1;N_k]$ define the finite state systems:
\begin{equation}
    \label{sigmaik}
\Sigma_{i}(k) : 
\left\{
\begin{array}
{l}
x_i(t+1)\in f_i^k(x_i(t),u_i^k(t),w_i^k(t)),\\
x_i(0) \in X_{i,0},
x_i(t) \in X_i, \\
u_i^k(t) \in U_i^k,
w_i^k(t) \in W_i^k, t\in\mathbb{N},
\end{array}
\right.
\end{equation}
where: $u_i^k(.)=(h_{1,i}(x_1(.)),...,h_{i-1,i}(x_{i-1}(.)),h_{i+1,i}(x_{i+1}(.)),$ $..., h_{N_k,i}(x_{N_k}(.)))$; $w_i^k(.)=(w_{N_k+1}(.),w_{N_k+2}(.),...,w_{N}(.))$ and $w_j:\mathbb{N} \rightarrow U_j, j\in [N_k+1;N]$; $U_i^k=\Pi_{j\in [1;N_k]}U_j$; $W_i^k=\Pi_{j\in [N_k+1;N]}U_j$,
and $f_i^k : X_i \times U_i^k \times W_i^k \rightarrow 2^{X_i}$ is defined by $f_i^k(x_i,u_i^k,w_i^k)
:=f_i(x_i,u'_i)$, 
with $u'_i=(u_i^k,w_i^k)$.     
\end{definition}
In the sequel, we refer to $u_i^k$ as the internal input of $\Sigma_i(k)$ and to $w_i^k$ as the external input of $\Sigma_i(k)$.  
We are now ready to define sequence of NoS $\{\mathcal{N}(k)\}_{k\in\mathbb{N}}$ by
\begin{equation}
\label{NdiK}
\mathcal{N}(k)=(\mathcal{V}(k),\mathcal{E}(k),h^k),
\end{equation}
where: $\mathcal{V}(k)=\{\Sigma_{1}(k),\Sigma_{2}(k),...,\Sigma_{N_k}(k)\}$, $(\Sigma_{i}(k),\Sigma_{j}(k))\in\mathcal{E}(k)$ if and only if $(v_i,v_j)\in E_k$, $h_{j,i}^k=h_{j,i}$ for all $(v_j,v_i)\in E_{k}$. Note that $\mathcal{N}(k)$ is a sub-NoS of NoS $\mathcal{N}(k')$ for $k<k'$, and that NoS $\mathcal{N}(0)$ exhibits a decentralized control architecture because controller $C=\Sigma_2$ shares information only with plant $P=\Sigma_1$. It is readily seen that 
\begin{proposition}
\label{prop1}
$\exists K\in\mathbb{N}$ s.t. $\mathcal{N}(k)=\mathcal{N}, \forall k\geq K$.
\end{proposition}
We now solve Problem \ref{problem1} for NoS $\mathcal{N}(k)$, $k=1,2,...$,  by making use of Algorithm \ref{alg1}. We assume not to have  knowledge of solution at step $k-1$, when designing solution at step $k$. 
Define operator $\mathbb{A}$ associated to Algorithm \ref{alg1} by:
\begin{equation}
\label{aaa}
(C',X'_{2,f},h'_{2,1},\mathcal{R}',T^c,U'_{2,not})=\mathbb{A}(\mathcal{N}',Q',\theta'),
\end{equation}
if inputs and outputs of Algorithm \ref{alg1} are $(\mathcal{N}',Q',\theta')$ and $(C',X'_{2,f},h'_{2,1},\mathcal{R}',T^c,U'_{2,not})$, respectively. 
In order to solve Problem \ref{problem1} for NoS $\mathcal{N}(k)$ by means of Algorithm \ref{alg1}, we first need to associate an appropriate NoS $\mathcal{N}'(k)$ to 
$\mathcal{N}(k)$, which is 
in the same form as $\mathcal{N}'$. 
\begin{definition}
    \label{defsigma'jk}
For any $k\in \mathbb{N}$ and any $j\in [1;3]$ define the finite state systems:
\begin{equation}
\label{sigmaj'k}
\Sigma'_{j}(k) : 
\left\{
\begin{array}
{l}
x_j^{\prime,k}(t+1)\in f_j^{\prime,k}(x^{\prime,k}_j(t),u^{\prime,k}_j(t),w^{\prime,k}_j(t)),\\
x^{\prime,k}_j(0) \in X^{\prime,k}_{j,0},
x^{\prime,k}_j(t) \in X^{\prime,k}_j, \\ 
u^{\prime,k}_j(t) \in U^{\prime,k}_j,
w^{\prime,k}_j(t) \in W^{\prime,k}_j, t\in\mathbb{N},
\end{array}
\right.
\end{equation}
where:
$x^{\prime,k}_1(.)$ is a state trajectory of $\Sigma_1(k)$ in (\ref{sigmaik}), 
$x^{\prime,k}_2(.)$ is a state trajectory of $\Sigma_2(k)$ in (\ref{sigmaik}), 
$x^{\prime,k}_3(.)$ is a state trajectory of the system obtained by interconnecting system $\Sigma_i(k)$ for $i\in [3;N_k]$ in (\ref{sigmaik});
$X^{\prime,k}_1=X_1$, 
$X^{\prime,k}_2(.)=X_2$, 
$X^{\prime,k}_3=\Pi_{i\in [3;N_k]}X_i$; 
internal inputs:
$u^{\prime,k}_1(.)=u^k_1(.)$, 
$u^{\prime,k}_2(.)=u^k_2(.)$,
$u^{\prime,k}_3(.)=(u^k_3(.),u^k_4(.),...,u^k_{N_k}(.))$, where $u^k_i(.)$ are as in Definition \ref{defSigmaIK};
external inputs: 
$w^{\prime,k}_1(.)$, $j\in [1;3]$, is a vector where entries are all $u_i(.)$ such that $\Sigma_i\notin \mathcal{V}(k)$ and $(\Sigma_i,\Sigma_1)\in \mathcal{E}(k)$,
$w^{\prime,k}_2(.)=w_{null}$ is a dummy external input (meaning that there is no external input acting on the system considered), and
$w^{\prime,k}_3(.)$ is a vector where entries are all $u_i(.)$ such that $\Sigma_i\notin \mathcal{V}(k)$ and $(\Sigma_i,\Sigma_{i'})\in \mathcal{E}(k)$ for some $i' \in [3;N_k]$;
$U^{\prime,k}_1=U^k_1$,
$U^{\prime,k}_2=U^k_2$, and
$U^{\prime,k}_3=\Pi_{i\in[3;N_k]}U_i$;
$W^{\prime,k}_1$ is the cross product of sets $U_i$ such that $\Sigma_i\notin \mathcal{V}(k)$ and $(\Sigma_i,\Sigma_1)\in \mathcal{E}(k)$,
$W^{\prime,k}_2=\{w_{dummy}\}$, and
$W^{\prime,k}_3$ is the cross product of sets $U_i$ such that $\Sigma_i\notin \mathcal{V}(k)$ and $(\Sigma_i,\Sigma_{i'})\in \mathcal{E}(k)$ for some $i' \in [3;N_k]$,
function $f^{\prime,k}_j : X^{\prime,k}_j \times U^{\prime,k}_j \times W^{\prime,k}_j \rightarrow 2^{X^{\prime,k}_j}$ is defined by 
$f^{\prime,k}_j(x^{\prime,k}_j,u^{\prime,k}_j,w^{\prime,k}_i)
:=f^{\prime,k}_j(x^{\prime,k}_j,u')$, with $u'=u^{\prime,k}_j,w^{\prime,k}_j)$.
\end{definition}
\begin{definition}
\label{defN''(k)}
Given NoS $\mathcal{N}(k)$ define NoS
$\mathcal{N}'(k):=(\mathcal{V}'(k),\mathcal{E}'(k),h^{\prime,k})$, where
$\mathcal{V}'(k)=\{\Sigma'_1(k),\Sigma'_2(k),\Sigma'_3(k)\}$;
$\mathcal{E}'(k)$ is composed of edge $(\Sigma'_2(k),\Sigma'_1(k))$ and edges: 
$(\Sigma'_1(k),\Sigma'_2(k))$, if $(\Sigma_1(k),\Sigma_2(k))\in\mathcal{E}(k)$; 
$(\Sigma'_1(k),\Sigma'_3(k))$, if $\exists i \in [3;N_k]$ s.t.
$(\Sigma_1(k),\Sigma_i(k))\in\mathcal{E}(k)$; 
$(\Sigma'_3(k),\Sigma'_1(k))$, if $\exists i \in [3;N_k]$ s.t. $ (\Sigma_i(k),\Sigma_1(k))\in\mathcal{E}(k)$;
$(\Sigma'_3(k),\Sigma'_2(k))$, if $\exists i \in [3;N_k]$ s.t.  $(\Sigma_i(k),\Sigma_2(k))\in\mathcal{E}(k)$; connection function $h^{\prime,k}$, comprising collection of functions $h^{\prime,k}_{j,i}$, is derived by connection functions $h_{j,i}$, accordingly.
\end{definition}
Define sequences of controllers $\{C(k)\}_{k\in \mathbb{N}}$, final states $\{X_{2,f}(k)\}_{k\in \mathbb{N}}$, connection functions $\{h^{\prime,k}_{2,1}\}_{k\in \mathbb{N}}$, relations $\{\mathcal{R}(k)\}_{k\in \mathbb{N}}$ and sets $\{U'_{2,not}(k)\}_{k\in \mathbb{N}}$, as follows:
\begin{equation}
\label{quasincr}
\begin{array}
{l}
(C(k),X_{2,f}(k),h^{\prime,k}_{2,1},\mathcal{R}(k),T^{c(k)},U'_{2,not}(k))=\\
\mathbb{A}(\mathcal{N}'(k),Q,\theta),k\in\mathbb{N}.
\end{array}
\end{equation}
We can now give the following result.
\begin{theorem}
\label{thquasincr}
For any $k\in\mathbb{N}$, quintuple 
$(C(k),X_{2,f}(k),h^k_{2,1},\mathcal{R}(k),T^{c(k)},U'_{2,not}(k))$ in (\ref{quasincr})
 solves Problem \ref{problem1} for NoS $\mathcal{N}$, with $(C,X_{2,f},h_{2,1},\mathcal{R},U'_{2,not})=(C(k),X_{2,f}(k),h^k_{2,1},\mathcal{R}(k),$ $U'_{2,not}(k))$. 
\end{theorem}
\begin{proof}
Consider any state trajectory $\mathbf{x}(.)=(x_1(.),x_2(.),...,x_N(.))$ of $\Sigma_{[1;N]}$ with controller $C=C(k)$ and initial state $\mathbf{x}(0) \in \mathcal{R}(k) \times X_{3,0} \times ... \times X_{N,0}(k)$. Let $l+1$ be the length of $\mathbf{x}(.)$. Any $x_i(.)$ of $\mathbf{x}(.)$ satisfies (\ref{ith-plant}) for some $u_i(.)$ in the form of (\ref{ui}).
By definition of $\Sigma_i(k)$, function $x_i(.)$ is a state trajectory of $\Sigma_i(k)$ with internal input $u^k_i(.)$ and external input $w^k_i(.)$. As a consequence, function 
$\mathbf{x}^k(.)=(x_1(.),x_2(.),...,x_{N_k}(.))$ is a state trajectory of $\Sigma_{[1;N_k]}$.
By Definition \ref{defN''(k)}, triplet $(x^{\prime,k}_1(.),x^{\prime,k}_2(.),x^{\prime,k}_3(.))$, with 
$x^{\prime,k}_1(.)=x_1(.)$, $x^{\prime,k}_2(.)=x_2(.)$ and 
$x^{\prime,k}_3(.)=(x_3(.),x_4(.),...,x_{N_k}(.))$, is a state trajectory of 
$\Sigma(\mathcal{N}'(k))$, where $\Sigma(\mathcal{N}'(k))$ denotes the system obtained by the interconnection of $\Sigma'_1(k)$, $\Sigma'_2(k)$ and $\Sigma'_3(k)$. 
By (\ref{quasincr}), $(x^{\prime,k}_1(0),x^{\prime,k}_2(0))\in\mathcal{R}(k)$, $x^{\prime,k}_2(l)\in X_{2,f}(k)$ and there exists a word $q_0\,q_1\,...\,q_l\in Q$ with length $l+1$, such that $\mathbf{d}(x^{\prime,k}_1(t),q_t)\leq \theta,\forall t\in[0;l]$, which implies, since  $x^{\prime,k}_1(.)=x_1(.)$, that inequality (\ref{ineq}) holds. 
\end{proof}
The following result shows that as long as index $k$ increases, the solution obtained gets better, in terms of the part of the specification that can be enforced on the plant.
\begin{theorem}
\label{thBehnonincr}
For any $k\in\mathbb{N}$ and for any quadruplet $(C(k),X_{2,f}(k),h^k_{2,1},\mathcal{R}(k))$ extracted from (\ref{quasincr}), the following inclusion holds: 
$
\mathcal{Q}
(\Sigma_{[1;N]},C(k),X_{2,f}(k),h_{2,1}(k),\mathcal{R}_0(k))\subseteq \\
\mathcal{Q}(\Sigma_{[1;N]},C(k+1),X_{2,f}(k+1),h_{2,1}(k+1),\mathcal{R}_0(k+1))
$, 
where $\mathcal{R}_0(k')=\mathcal{R}(k') \times X_{3,0} \times ... \times X_{N,0}$, with $k'=k,k+1$. Moreover, there exists $K\in\mathbb{N}$ such that for any $k\geq K$
\begin{equation}
\label{stat2}
\begin{array}
{l}
\mathcal{Q}
(\Sigma_{[1;N]},C(k),X_{2,f}(k),h_{2,1}(k),\mathcal{R}_0(k))
=\\\mathcal{Q}(\Sigma_{[1;N]},C,X_{2,f},h_{2,1},\mathcal{R}).
\end{array}
\end{equation}
\end{theorem}
\begin{proof}
Consider any word 
\begin{equation}
\label{word}
q_0\,q_1\,...\,q_l\in \mathcal{Q} (\Sigma_{[1;N]},C(k),X_{2,f}(k),h_{2,1}(k),\mathcal{R}_0(k)).
\end{equation}
By definition of $\mathcal{Q}$, there exists a state trajectory $\mathbf{x}(.)$ of $\Sigma_{[1;N]}$ with initial state $\mathbf{x}(0)\in \mathcal{R}_0(k)$, length $l+1$, and final state $x_2(l)\in X_{2,f}(k)$ of the controller $C(k)$ such that (\ref{ineq}) holds. 
Define
$u^k(.)=(u^k_1(.),u^k_2(.),...,u^k_{N_k}(.))$ and $
w^k(.)=(w^k_{N_k+1}(.),w^k_{N_k+2}(.),...,w^k_{N}(.))$, 
where $u^k_i(.)$ and $w^k_i(.)$ are as in Definition \ref{defSigmaIK}.
By definition of sequence $\{\mathcal{N}(k)\}_{k\in\mathbb{N}}$, all components of $u^k(.)$ are internal to $\Sigma_{[1;N_{k+1}]}$, some components of $w^k(.)$, group them in $w_{ext}^k(.)$, are external to $\Sigma_{[1;N_{k+1}]}$ and the remaining components of $w^k(.)$, group them in $w_{int}^k(.)$, are internal to $\Sigma_{[1;N_{k+1}]}$.
Hence, $u^{k+1}(.)$, collecting all components of $u^k(.)$ and of $w_{int}^k(.)$, is the internal input of $\Sigma_{[1;N_{k+1}]}$ and, $w^{k+1}(.)=w_{ext}^k(.)$ is the external input of $\Sigma_{[1;N_{k+1}]}$. 
By Algorithm \ref{alg1}, controller $C(k)$ is such that $\mathbf{x}(.)$ satisfies condition (\ref{ineq}) with word in (\ref{word}) irrespective of any external input of $\Sigma_{[1;N_{k}]}$ and in particular, for the internal input $w_{int}^k(.)$ of $\Sigma_{[1;N_{k}]}$, and, irrespective of any external input of $\Sigma_{[1;N_{k+1}]}$.
This implies that state trajectory $\mathbf{x}(.)$ satisfies condition (\ref{ineq}) with word in (\ref{word}), with internal input $u^{k+1}(.)$ and irrespective of any external input of 
$\Sigma_{[1;N_{k+1}]}$. 
Hence, word $q_0,q_1\,...\,q_l\in \mathcal{Q} (\Sigma_{[1;N]},C(k+1),X_{2,f}(k+1),h_{2,1}(k+1),\mathcal{R}_0(k+1))$ from which, the first part of the statement holds. 
As for the second part of the statement, by Proposition \ref{prop1}, there exists $K\in\mathbb{N}$ such that $\mathcal{N}(k)=\mathcal{N}$ for all $k\geq K$.  
Hence, for any $k\geq K$, (\ref{quasincr}) rewrites as  
$(C(k),X_{2,f}(k),h^k_{2,1},\mathcal{R}(k),T^{c(k)},U'_{2,not}(k))=\mathbb{A}(\mathcal{N}'',Q,\theta)$, 
which by (\ref{aaa}), implies $C(k)=C'$, $X_{2,f}(k)=X'_{2,f}$, $h^k_{2,1}=h'_{2,1}$, and $\mathcal{R}(k)=\mathcal{R}'$; by definition of $\mathcal{R}_0(k)$ and $\mathcal{R}$, equality $\mathcal{R}(k)=\mathcal{R}$ implies $\mathcal{R}_0(k)=\mathcal{R}$. Since 
$C=C'$, $X_{2,f}=X'_{2,f}$, $h_{2,1}=h'_{2,1}$ and $\mathcal{R}=\mathcal{R}'$, the result follows.
\end{proof}
A drawback of Theorems \ref{thquasincr} and \ref{thBehnonincr}
 is that at each step $k$, Algorithm \ref{alg1} computes tuple $(C(k),X_{2,f}(k),h^k_{2,1},$ $\mathcal{R}(k),U'_{2,not}(k))$ without gaining information from the solution computed at the previous steps $k'<k$. We now  incorporate solutions found at step $k-1$ when finding solutions at step $k'$. 
The main idea is to replace recursive equations in (\ref{quasincr}) by the following ones:
\begin{equation}
\label{incr}
\begin{array}
{l}
(C(0),X_{2,f}(0),h^0_{2,1},\mathcal{R}(0),T^c,U'_{2,not}(0))=\mathbb{A}(\mathcal{N}'(0),Q,\theta);\\
(C(k+1),X_{2,f}(k+1),h^{k+1}_{2,1},\mathcal{R}(k+1),T^{c(k+1)},\\U'_{2,not}(k+1))
=
\partial\mathbb{A}(\mathcal{N}'(k),Q,\theta, T^{c(k)},U'_{2,not}(k)),k \in \mathbb{N},
\end{array}
\end{equation}
where the new operator $\partial\mathbb{A}$ comes into play. The difference w.r.t. operator $\mathbb{A}$ is that here $\partial\mathbb{A}$ depends on $U'_{2,not}(k)$ collecting inputs $u'_2\in U'_2$ not providing solution in Algorithm \ref{alg1}) and also on $T^{c(k)}$ and hence, on the solution $(C(k),X_{2,f}(k),h^k_{2,1},\mathcal{R}(k),U'_{2,not}(k))$ found at the previous step $k$, while $\mathbb{A}$ does not (compare (\ref{quasincr}) and (\ref{incr})). 
Operator $\partial\mathbb{A}$ is defined through Algorithm \ref{alg2}. It is readily seen that 
Algorithm \ref{alg2} terminates in a finite number of steps.
The following result shows that the outcome of recursive equations in (\ref{quasincr}) and in (\ref{incr}) coincide.
\incmargin{1em}
\restylealgo{boxed}\linesnumbered
\begin{algorithm*}[t]
\label{alg2}
\SetLine
\textbf{input}:  
NoS $\mathcal{N}'(k)$; specification $Q$; accuracy $\theta$; Transition system $T^{c(k)}$; Set $U'_{2,not}(k)$\;
\textbf{init}: 
$T^{c(k+1)}=(X^{c(k+1)},X^{c(k+1)}_0,U^{c(k+1)},
\xrightarrow[\text{ }\text{ }c(k+1)\text{ }\text{ }]{\text{}},X^{c(k+1)}_m,Y_q,H^{c(k+1)}):=T^{c(k)}$; $h:X_2(k+1)\rightarrow U_{2,1}$ empty function; $\mathcal{R}(k+1):=\varnothing$; 
$U'_{2,not}(k+1):=\varnothing$ \;
\ForEach {$u'_2=(u'_{1,2},u'_{3,2})\in U'_{2,not}(k)\times (\Pi_{i\in [N_k+1;N_{k+1}]}U_{i,2}(k))$}{
\ForEach {$u'_{2,1}\in U'_{2,1}$ \text{ s.t.} 
\begin{equation}
\label{eqAlg2}
\begin{array}{c}
\forall (x'_1,x'_3,x_q) \in X^{c(k+1)} \cap 
((((h^{\prime,k+1}_{1,2})^{-1})( u'_{1,2})) \times (((h^{\prime,k+1}_{3,2})^{-1})(u'_{3,2})) \times X_q ),
\forall (w_1,w_3) \in W^{\prime,k+1}_1 \times W^{\prime,k+1}_3,\\
\forall x'_1 \in f^{\prime,k+1}_1(x'_1,(u'_{2,1},h^{\prime,k+1}_{3,1}(x'_3)),w_1),\forall x'_3\in f^{\prime,k+1}_3(x'_3,h^{\prime,k+1}_{1,3}(x'_1),w_3),
\exists x_q 
\xrightarrow[\text{ }\text{ }q\text{ }\text{ }]{\text{}}
x_q^+ \text{ s.t. } \mathbf{d}'(x^+_1,H_q(x_q^+))\leq \theta
 \end{array}  
\end{equation}
}{
$X^{c(k+1)}:=X^{c(k+1)} \cup \{(x^+_1,x^+_3,x_q^+)\};H^{c(k+1)}((x^+_1,x^+_3,x_q^+)):=H_q(x_q^+);
h((x'_1,x'_3,x_q)):=u'_{2,1}$\;
$
\xrightarrow[\text{ }\text{ }c(k+1)\text{ }\text{ }]{\text{}}:=\xrightarrow[\text{ }\text{ }c(k+1)\text{ }\text{ }]{\text{}} \cup \{((x'_1,x'_3,x_q),(h^{\prime,k}_{1,2}(x^+_1),h^{\prime,k}_{3,2}(x^+_3)),(x^+_1,x^+_3,x_q^+))\}$\;
}
\If{
$\nexists u'_{2,1}\in U'_{2,1}$ \text{ s.t. (\ref{eqAlg2}) holds}}
{$U'_{2,not}(k+1):=U'_{2,not}(k+1) \cup \{u'_2\}$\;}
}
$T^{c(k+1)}:=\Trim(T^{c(k+1)})$\; 
\textbf{output} $(C(k+1),X_{2,f}(k+1),h^{k+1}_{2,1},\mathcal{R}(k+1),T^{c(k+1)},U'_{2,not}(k+1))$
specified by\;
controller $C(k+1)$ defined by $X_2(k+1):=X^{c(k+1)}$; 
$X_{2,0}(k+1):=X^{c(k+1)}_0$; 
$f_2^{k+1}((x'_1,x'_3,x_q),(u^+_{1,2},u^+_{3,2})):=\left\{
(x^+_1,x^+_3,x_q^+)\in X^{c(k+1)} \,|\, (x'_1,x'_3,x_q) 
\xrightarrow[\text{ }\text{ }c(k+1)\text{ }\text{ }]{(u^+_{1,2},u^+_{3,2})}
(x^+_1,x^+_3,x_q^+)\right\}$\;
$X'_{2,f}(k+1):=X^{c(k+1)}_m$; $h^{k+1}_{2,1}((x'_1,x'_3,x_q)):=h((x'_1,x'_3,x_q))$ for all $(x'_1,x'_3,x_q)\in X_2(k+1)$\; 
relation $\mathcal{R}(k+1)$ defined by $\mathcal{R}(k+1):=\{(x'_1,(x''_1,x''_3,x_q))\in X_1(k+1)\times X_2(k+1)\,|\,x'_1=x''_1 \}$.
\caption{Incremental control design of NoS $\mathcal{N}$.}
\end{algorithm*}

\begin{theorem}
    \label{lemma2}
For any $k\in\mathbb{N}$, 
\[
\begin{array}
{l}
\partial\mathbb{A}(\mathcal{N}'(k),Q,\theta, T^{c(k)},U'_{2,not}(k))=
\mathbb{A}(\mathcal{N}'(k),Q,\theta).
\end{array}
\]    
\end{theorem}

\begin{proof}
Since outputs of Algorithms \ref{alg1} and \ref{alg2} are derived from transition systems $T^c$ and $T^{c(k+1)}$, respectively, 
the result holds if $T^c$ in line 13 of Algorithm \ref{alg1} and $T^{c(k+1)}$ in line 12 of Algorithm \ref{alg2}, apart from different notation used, coincide. 
All entities defining transition systems $T^c$ and $T^{c(k+1)}$ are updated in line 7 of Algorithm \ref{alg1} and lines 5-6 of Algorithm \ref{alg2}, respectively; note that apart from different notation used for indicating entities in $T^c$ and $T^{c(k+1)}$, lines 7 and 5-6 coincide. 
By definition of set $U'_{2,not}(k)$, for any input $u'_2=(u'_{1,2},u'_{3,2})\in U'_2$ of $C'$ in line 5 of Algorithm \ref{alg1}, there exists an input $u'_{2,1}\in U'_{2,1}$ in line 6 of Algorithm \ref{alg1} enforcing some transitions in specification $Q$, if and only if there exists an input $u'_{2,1}\in U'_{2,1}$ in line 4 of Algorithm \ref{alg2} enforcing the same transitions in $Q$.
Hence, the result follows.
\end{proof}
The following results hold directly from Theorems \ref{thquasincr}, \ref{thBehnonincr}, \ref{lemma2}.
\begin{corollary}
\label{thincr}
For any $k\in\mathbb{N}$, quadruplet  (\ref{quadruplet}) 
 solves Problem \ref{problem1} for the original NoS $\mathcal{N}$, with 
 $(C,X_{2,f},h_{2,1},\mathcal{R})=(C(k),X_{2,f}(k),h^k_{2,1},\mathcal{R}(k))$, 
 where $(C(k),X_{2,f}(k),h^k_{2,1},\mathcal{R}(k))$ satisfies (\ref{incr}).
\end{corollary}
\begin{corollary}
\label{thincr1}
For any $k\in\mathbb{N}$ and for any quadruplet (\ref{quadruplet}), the following inclusion holds:
\[
\begin{array}{l}
\mathcal{Q}(\Sigma_{[1;N]},C(k+1),X_{2,f}(k+1),h^{k+1}_{2,1},\mathcal{R}_0(k+1))\subseteq \\
\mathcal{Q}(\Sigma_{[1;N]},C(k),X_{2,f}(k),h^k_{2,1},\mathcal{R}_0(k)), 
\end{array}
\]
 where $(C(k),X_{2,f}(k),h^k_{2,1},\mathcal{R}(k))$ and $(C(k+1),X_{2,f}k+1),$ $h^{k+1}_{2,1},\mathcal{R}(k+1))$ satisfy (\ref{incr}), and $\mathcal{R}_0(k)=\mathcal{R}(k) \times X_{3,0} \times ... \times X_{N,0}$. Moreover, there exists $K\in\mathbb{N}$ such that 
$\mathcal{Q}(\Sigma_{[1;N]},C(k),X_{2,f}(k),h^k_{2,1},\mathcal{R}_0(k))=
\mathcal{Q}(\Sigma_{[1;N]},C,X_{2,f},h_{2,1},\mathcal{R}_0)
$, for any $k \geq K$.
\end{corollary}
\section{Extension to NoS with multi-objectives}\label{sec6}
In this section we consider multiple plants, controllers and specifications. 
Consider NoS $\mathcal{N}=(\mathcal{V},\mathcal{E},h)$. Among systems in $\mathcal{V}$ we consider a collection of plants $\mathbf{P} =\{\mathbf{P}_1,\mathbf{P}_2,..., \mathbf{P}_{\mathbf{N}}\}\subseteq \mathcal{V}
$ 
and a collection of controllers 
$\mathbf{C} =\{\mathbf{C}_1,\mathbf{C}_2,..., \mathbf{C}_{\mathbf{N}}\} \subseteq \mathcal{V}
$
with $\mathbf{N}\leq N$. System $\mathbf{C}_{\mathbf{i}}$ is the controller associated with plant $\mathbf{P}_{\mathbf{i}}$. 
Sets of states and of initial states of $\mathbf{C}_{\mathbf{i}}$ and $\mathbf{P}_{\mathbf{i}}$
are denoted by $\mathbf{X}^{\mathbf{c}}_{\mathbf{i}}$ and $\mathbf{X}^{\mathbf{p}}_{\mathbf{i}}$, and by 
$\mathbf{X}^{\mathbf{c}}_{\mathbf{i,0}}$ and $\mathbf{X}^{\mathbf{p}}_{\mathbf{i,0}}$, respectively; sets of inputs by $\mathbf{U}^{\mathbf{c}}_{\mathbf{i}}$ and $\mathbf{U}^{\mathbf{p}}_{\mathbf{i}}$, respectively; 
state trajectories by $\mathbf{x}^{\mathbf{c}}_{\mathbf{i}}(.)$ and $\mathbf{x}^{\mathbf{p}}_{\mathbf{i}}(.)$, respectively. 
We suppose that $\mathbf{X}^{\mathbf{p}}_{\mathbf{i}} $ is metric with metric $\mathbf{d}_{\mathbf{i}}: \mathbf{X}^{\mathbf{p}}_{\mathbf{i}} \times \mathbf{X}^{\mathbf{p}}_{\mathbf{i}} \rightarrow \mathbb{R}^+_0 $. In the sequel we make the following
\begin{assumption}
\label{ass4}
Pair $(\mathbf{C}_{\mathbf{i}},\Sigma_j)\in\mathcal{E}$ if and only if $\Sigma_j=\mathbf{P}_{\mathbf{i}}$. 
For any $\Sigma_i\in \mathcal{V}$ and for any $\mathbf{i}\in [1;\mathbf{N}]$, there exists a directed path from $\Sigma_i$ to either $\mathbf{C}_{\mathbf{i}}$ or $\mathbf{P}_{\mathbf{i}}$.
\end{assumption}
Specification associated with $\mathbf{P}_{\mathbf{i}}$ is given as a regular language $\mathbf{Q}_{\mathbf{i}}$ defined over the alphabet $\mathbf{X}^{\mathbf{p}}_{\mathbf{i}}$ where $\mathbf{X}^{\mathbf{p}}_{\mathbf{i}}$ is the set of states of $\mathbf{P}_{\mathbf{i}}$. Let
$\mathbf{S}=\{(\mathbf{P}_{\mathbf{i}},\mathbf{Q}_{\mathbf{i}}), \mathbf{i}\in [1;\mathbf{N}] \}$.
The problem we address in this section naturally extends Problem \ref{problem1} as follows:
\begin{problem}
\label{problem1gen}
Find a set $\mathcal{R}\subseteq \mathbf{X}$ of initial states for the interconnected system $\Sigma_{[1;N]}$, and 
for any $\mathbf{i}\in [1;\mathbf{N}]$, for any accuracy $\mathbf{\theta}_{\mathbf{i}}\in\mathbb{R}^+_0$, find a controller $\mathbf{C}_{\mathbf{i}}$, a set of final states $\mathbf{X}^{\mathbf{c}}_{\mathbf{i},\mathbf{f}}\subseteq \mathbf{X}^{\mathbf{c}}_{\mathbf{i}}$ of $\mathbf{C}_{\mathbf{i}}$, and a connection function $\mathbf{h}_{\mathbf{i}}$ associated with edge 
$(\mathbf{C}_{\mathbf{i}},\mathbf{P}_{\mathbf{i}})\in \mathcal{E}$, 
such that for any state trajectory $\mathbf{x}(.)=(x_1(.),x_2(.),...,$ $x_N(.))$ of $\Sigma_{[1;N]}$
with length $\mathbf{l}+1$, with initial state 
$\mathbf{x}(0)=(x_1(0),x_2(0),...,x_N(0))\in\mathcal{R}$, there exist $\mathbf{l}_\mathbf{i}\in [\mathbf{l};\infty[$ and word $\mathbf{q}^{\mathbf{i}}_0\, \mathbf{q}^{\mathbf{i}}_1\,...\, \mathbf{q}^{\mathbf{i}}_{\mathbf{l}_{\mathbf{i}}} \in \mathbf{Q}_{\mathbf{i}}$ such that 
$\mathbf{x}^{\mathbf{c}}_{\mathbf{i}}(\mathbf{l}_{\mathbf{i}})\in \mathbf{X}^{\mathbf{c}}_{\mathbf{i},\mathbf{f}}$ and the following inequalities hold:
\begin{equation}
    \label{ineqgen}
    \mathbf{d}_{\mathbf{i}}(\mathbf{x}^{\mathbf{p}}_{\mathbf{i}}(t),\mathbf{q}^{\mathbf{i}}_t)\leq \mathbf{\theta}_{\mathbf{i}},\forall t\in[0;\mathbf{l}_\mathbf{i}],  \mathbf{i}\in[1;\mathbf{N}].
\end{equation}
\end{problem}
\medskip
Problem above can be solved with a centralized approach. Indeed, one can define a plant $P$, obtained by interconnecting all plants $\mathbf{P}_{\mathbf{i}}\in\mathbf{P}$, a controller $C$, obtained by interconnecting all controllers $\mathbf{C}_{\mathbf{i}}\in\mathbf{C}$, and a system obtained by interconnecting all
remaining systems in the NoS $\mathcal{N}$, to then apply the results presented in Sections \ref{sec4} and \ref{sec5} to the resulting NoS. The main drawback of this approach is with computational complexity, as discussed in the next section. For this reason, in the following we propose an alternative approach. This approach consists in decomposing Problem \ref{problem1gen} into $\mathbf{N}$ sub-problems which coincide with Problem \ref{problem1} with $C=\mathbf{C}_{\mathbf{i}}$, $P=\mathbf{P}_{\mathbf{i}}$ and $Q=\mathbf{Q}_{\mathbf{i}}$. The solution to Problem \ref{problem1gen} is given by Algorithm \ref{alg3}. It is readily seen that 
Algorithm \ref{alg3} terminates in a finite number of steps. Moreover:
\begin{theorem}
\label{thlast}
Output of Algorithm \ref{alg3} solves Problem \ref{problem1gen}.
\end{theorem}
\begin{proof}
(Sketch.) 
For solving Problem \ref{problem1gen}, we associate a sequence $\{\mathcal{N}(k)\}_{k\in\mathbb{N}}$ to each pair $(\mathbf{C}_{\mathbf{i}},\mathbf{P}_{\mathbf{i}})$  by setting $\Sigma_1=\mathbf{P}_{\mathbf{i}}$ and $\Sigma_2=\mathbf{C}_{\mathbf{i}}$. 
For any $k\in\mathbb{N}$ and $\mathbf{i}\in [1;\mathbf{N}]$, depending on the topology of the NoS $\mathcal{N}(k)$, input $\mathbf{h}_{\mathbf{j}}(\mathbf{x}^{\mathbf{c}}_{\mathbf{j}}(.))$ with $\mathbf{j} \neq \mathbf{i}$ may be external to $P'=\mathbf{P}_{\mathbf{i}}=\Sigma_1(k)$ or not. Algorithms \ref{alg1} and \ref{alg2} provide solution to sub-problems $(\mathbf{P}_{\mathbf{i}},\mathbf{Q}_{\mathbf{i}})$ in a way such that plant $P'=\mathbf{P}_{\mathbf{i}}=\Sigma_1(k)$ satisfies the specification $Q'=\mathbf{Q}_{\mathbf{i}}$, irrespective of evolution of external inputs of $\mathcal{N}(k)$, possibly including $\mathbf{h}_{\mathbf{j}}(\mathbf{x}^{\mathbf{c}}_{\mathbf{j}}(.))$, $\mathbf{j} \neq \mathbf{i}$, which ranges arbitrarily in the whole set $U_{\phi(\mathbf{C}_\mathbf{i}),\phi(\mathbf{P}_\mathbf{i})}$. When considering NoS $\mathcal{N}(k)$ with controllers $\mathbf{C}_{\mathbf{i}}$ and connection functions $\mathbf{h}_{\mathbf{i}}$, starting from any initial state $\mathbf{x}(0)=(x_1(0),x_2(0),...,x_N(0))\in \mathcal{R}$,
the corresponding state trajectory $\mathbf{x}^{\mathbf{p}}_{\mathbf{i}}(.)$ of $\mathbf{P}_{\mathbf{i}}$ satisfies (\ref{ineqgen}) because each possible external input $\mathbf{h}_{\mathbf{j}}(\mathbf{x}^{\mathbf{c}}_{\mathbf{j}}(.))$ of $\Sigma'_1=\mathbf{P}_{\mathbf{i}}$ in $\mathcal{N}(k)$ ranges in a (not necessarily proper) subset of $U_{\phi(\mathbf{C}_\mathbf{i}),\phi(\mathbf{P}_\mathbf{i})}$. 
Hence, the result follows.
\end{proof}
\incmargin{1em}
\restylealgo{boxed}\linesnumbered
\begin{algorithm}[t]
\label{alg3}
\SetLine
\textbf{input:}\,
NoS $\mathcal{N}(k)$; set $\mathbf{S}$; Accuracy $\theta_{\mathbf{i}}$\;
define 
$\phi: \mathbf{C} \cup \mathbf{P} \rightarrow [1;\mathbf{N}]$ by $\phi(\mathbf{C}_\mathbf{i})=i$ if $\mathbf{C}_\mathbf{i}=\Sigma_i$ and $\phi(\mathbf{P}_\mathbf{i})=i$ if $\mathbf{P}_\mathbf{i}=\Sigma_i$\;
\ForEach {$\mathbf{i} \in [1;\mathbf{N}]$}{
define 
\[
\begin{array}
{l}
\mathbf{C}_{\mathbf{i}} : 
\left\{
\begin{array}
{l}
\mathbf{x}^{\mathbf{c}}_{\mathbf{i}}(t)\in \mathbf{X}^{\mathbf{c}}_{\mathbf{i,0}}=\mathbf{X}^{\mathbf{c}}_{\mathbf{i}}=
U_{\phi(\mathbf{C}_\mathbf{i}),\phi(\mathbf{P}_\mathbf{i})}
,\\
\mathbf{u}^{\mathbf{c}}_{\mathbf{i}}(t)=
\mathbf{x}^{\mathbf{c}}_{\mathbf{i}}(t),
\end{array}
\right.
\\
\mathbf{h}_{\mathbf{i}}(\mathbf{x}^{\mathbf{c}}_{\mathbf{i}}(t))=\mathbf{x}^{\mathbf{c}}_{\mathbf{i}}(t), 
t\in\mathbb{N},
\end{array}
\]
}
update NoS $\mathcal{N}(k)$ with the collection of controllers $\mathbf{C}_{\mathbf{i}}$ and connection functions $\mathbf{h}_{\mathbf{i}}$\;
\ForEach {$\mathbf{i} \in [1;\mathbf{N}]$}{
find solution $(\mathbf{C}_{\mathbf{i}},\mathbf{X}^{\mathbf{c}}_{\mathbf{i},\mathbf{f}},\mathbf{h}_{\mathbf{i}},\mathcal{R}_{\mathbf{i}})$
to Problem \ref{problem1} for NoS $\mathcal{N}(k)$ with inputs $P=\mathbf{P}_{\mathbf{i}}$, $Q=\mathbf{Q}_{\mathbf{i}}$ and $\theta=\mathbf{\theta}_{\mathbf{i}}$ }
\textbf{output:} $(\mathbf{C}_{\mathbf{i}},\mathbf{X}^{\mathbf{c}}_{\mathbf{i},\mathbf{f}},\mathbf{h}_{\mathbf{i}})$, $\mathbf{i}\in [1;\mathbf{N}]$;
$\mathcal{R}=\{\mathbf{x}\in\mathbf{X}\,|\,(\mathbf{x}^{\mathbf{p}}_{i},\mathbf{x}^{\mathbf{c}}_{i})\in \mathcal{R}_\mathbf{i},\, \mathbf{i} \in [1;\mathbf{N}]\}$.
\caption{Control design of NoS with multi-objectives.}
\end{algorithm}

\section{Time Computational Complexity Analysis}\label{sec7}
We start with the following preliminary result.
\begin{proposition}
\label{propTCC1}
Time Computational Complexity (TCC) of Algorithm \ref{alg1} is 
\begin{equation}
    \label{StatTCC1}
O(\vert U'_2 \vert\,
\vert U'_{2,1} \vert \,
\vert W_1 \vert \,
\vert W_3 \vert \,
2^{2\vert X'_1 \vert} \,
2^{2\vert X'_3 \vert} \,
\vert 
\xrightarrow[\text{ }\text{ }q\text{ }\text{ }]{\text{}}
\vert ).
\end{equation}
\end{proposition}
\begin{proof}
By lines 10-11 of Algorithm \ref{alg1}, TCC is given by 
$O(
\vert U'_2  \vert 
\vert U'_{2,1}  \vert 
\vert X'_1  \vert 
\vert X'_3  \vert 
\vert X_q  \vert 
\vert W_1  \vert 
\vert W_3 \vert 
2^{\vert X'_1 \vert}2^{\vert X'_3 \vert}
\vert 
\xrightarrow[\text{ }\text{ }q\text{ }\text{ }]{\text{}}\vert)$. Since for $\alpha\in\mathbb{N}$, $\alpha2^{\alpha}\leq 2^{2\alpha}$ 
and $\vert X_q \vert \leq \vert \xrightarrow[\text{ }\text{ }q\text{ }\text{ }]{\text{}} \vert$, the result follows.
\end{proof}
\medskip
In the sequel we evaluate TCC of the main results of the paper w.r.t. $N$, $N_k$ and $N_q=\vert \xrightarrow[\text{ }\text{ }q\text{ }\text{ }]{\text{}} \vert$. The next result evaluates TCC of the incremental approach without past information.
\begin{proposition}
\label{propTCC3}
For any $k\in\mathbb{N}$, TCC of Theorem \ref{thquasincr} is 
\begin{equation}
    \label{StatTCC3}
O\left(N_q\,2^{2^{N_k}+2N-N_k}
\right).
\end{equation}
\end{proposition}
\begin{proof}
By (\ref{StatTCC1}), we get $
U'_2=U_2^k,
U'_{2,1}=U_{2,1},
X'_1=X_1,
X'_3=\pi_{i\in [3;N_k]}X_i,
W_1=W_3=\pi_{i\in [N_k+1;N]}U_i$. 
By (\ref{ui}) we get $O(\vert U_2^k \vert)\sim 2^{N_k}$. Since $U_{2,1},X_1,2^{X_1}$ are independent from $N$, $N_k$ and $N_q$, we get $O(\vert U_{2,1} \vert) \sim O(\vert X_1 \vert) \sim O(\vert 2^{X_1} \vert) \sim 1$. Moreover, $O(\vert X'_3 \vert)\sim 2^{N_k}$, 
because of definition of $X^{\prime,k}_3$ in Definition \ref{defsigma'jk}, corresponding to $X'_3$ in Algorithm \ref{alg1}. We also get $O(\vert W_1 \vert ) \sim O(\vert W_3 \vert )\sim 2^{N- N_k}$, 
because of definition of $W^{\prime,k}_1$ and $W^{\prime,k}_3$  in Definition \ref{defsigma'jk}, corresponding to $W_1$ and $W_3$, respectively, in Algorithm \ref{alg1}.
Hence, since $O(\vert X_q \vert)\leq O(\vert \xrightarrow[\text{ }\text{ }q\text{ }\text{ }]{\text{}} \vert )$ holds by definition of $\mathcal{N}$, by Proposition \ref{propTCC1}, 
(\ref{StatTCC1}) rewrites as (\ref{StatTCC3}). 
\end{proof}
The next result evaluates TCC of the incremental approach with past information.
\begin{proposition}
\label{propTCC4}
For any $k\in\mathbb{N}$, TCC of Corollary \ref{thincr} is 
\begin{equation}
    \label{StatTCC4}
O\left(N_q\,
2^{2N}(2^{2^{N_k}-N_k}-
2^{2^{N_{k-1}}-N_{k-1}})
\right).
\end{equation}
\end{proposition}
\begin{proof}
First of all note that $N_k\geq N_{k-1}$ from which $2^{2^{N_k}+N}-2^{2^{N_{k-1}}+N}\geq 0$. From Proposition \ref{propTCC3}, TCC in computing $T^{c(k)}$ is given by (\ref{StatTCC3}). Since Algorithm \ref{alg2} used in Corollary \ref{thincr} computes $T^{c(k+1)}$ given $T^{c(k)}$, TCC involved in Algorithm \ref{alg2} at step $k$ is given by (\ref{StatTCC4}).
\end{proof}
Some final remarks follow. 
TCC for solving Problem 1 in a centralized setting and by using whichever algorithm, is double-exponential in $N$ 
because the number of (needed to-be-explored) transitions of nondeterministic system $\Sigma_{[1;N]}$, composed of $N$ subsystems, scales as $2^{2^N}$.
TCC of Propositions \ref{propTCC3} and \ref{propTCC4} is exponential with $N$ and double-exponential with $N_k$. When $N$ is arbitrarily large and $N_k<<N$, the gain in terms of TCC w.r.t. centralized approaches is evident. When $N_k \sim N$, meaning that NoS $\mathcal{N}(k)$ is approaching closer and closer to NoS $\mathcal{N}$ (see Proposition \ref{prop1}), there may be no substantial gain in the approach we propose. Nevertheless, our approach is possibly able to find a solution for sufficiently small $k$, whereas centralized approaches may fail due to possible TCC burden.
TCC of Proposition \ref{propTCC4} is lower than TCC of Proposition \ref{propTCC3} because Corollary \ref{thincr} at step $k$ makes use of information about transition system $T^{c(k-1)}$ computed at step $k-1$, whereas Theorem \ref{thquasincr} does not. When multiple plants and controllers are considered, TCC for solving Problem \ref{problem1gen}  is $\mathbf{N}$ times TCC for solving Problem \ref{problem1}. Moreover, Algorithm \ref{alg3} allows using parallel computing architectures. 
When these computing architectures are used, resulting TCC for solving Problem \ref{problem1gen} coincides with TCC for solving Problem \ref{problem1}, which is important for scalability purposes. 

\section{Example}\label{sec8}
We consider the problem of regulating the temperature in a circular building composed of $\mathbf{N}\geq 3$ rooms, each one equipped with a heater. 
Each room $\mathbf{i}$ is adjacent to rooms $\mathbf{i}-1$ and $\mathbf{i}+1$; room $\mathbf{N}+1$ corresponds to room $1$. 
Dynamics of temperature $x_{\mathbf{i}}$ of each room $\mathbf{i}$ is the same and described by $\mathbf{P}_{\mathbf{i}}$ as in (\ref{ith-plant}), where 
$X_{\mathbf{i},0}=\{19\}$, $X_{\mathbf{i}}=[15;22]$, $V_{\mathbf{i}}=\{0,1,2\}$,
$U_{\mathbf{i}}=V_{\mathbf{i}}\times X_{\mathbf{i}-1} \times X_{\mathbf{i}+1}$, $x_{\mathbf{i}}(t)\in X_{\mathbf{i}}$ is the state and $v_\mathbf{i}(t)\in V_{\mathbf{i}}$ is the control input at time $t\in [0;5]$. 
Here, each time interval from $t$ to $t+1$ has four hours duration. Hence, temperatures in the rooms are monitored and controlled in a full day. 
Map $f_\mathbf{i}$ is given by $f_\mathbf{i}(x_\mathbf{i},u_\mathbf{i})=\{x_\mathbf{i}+\phi_\mathbf{i}(v_\mathbf{i},\Delta_\mathbf{i})\}$, where $u_\mathbf{i}=(v_\mathbf{i},x_{\mathbf{i}-1},x_{\mathbf{i}+1})$, $\Delta_\mathbf{i}=(x_{\mathbf{i}+1}-x_\mathbf{i})+(x_{\mathbf{i}-1}-x_\mathbf{i})$ and $\phi_\mathbf{i}(v_\mathbf{i},\Delta_\mathbf{i})$ is detailed in Table \ref{table1}.
\begin{table}[]
    \centering
\begin{eqnarray}
\begin{tabular}
     {|r|r|r|}
     \hline
$v_\mathbf{i}$ & $\Delta_\mathbf{i}$ & $\phi_\mathbf{i}$  \\     
\hline    
$0$ & $-4$ & $-2$   \\
$0$ & $-3$ & $-2$   \\
$0$ & $-2$ & $-1$   \\
$0$ & $-1$ & $-1$  \\
$0$ & $0$  & $-1$  \\
$0$ & $1$  & $0$  \\
$0$ & $2$  & $0$  \\
$0$ & $3$  & $0$  \\
$0$ & $4$  & $1$  \\
\hline
\end{tabular}
\notag &
\begin{tabular}
     {|r|r|r|}
     \hline
$v_\mathbf{i}$ & $\Delta_\mathbf{i}$ & $\phi_\mathbf{i}$ \\     \hline    
$1$ & $-4$ & $-2$   \\
$1$ & $-3$ & $-1$   \\
$1$ & $-2$ & $-1$  \\
 $1$ & $-1$ &  $0$ \\
 $1$ & $0$ &  $0$ \\
 $1$ & $1$ &  $0$   \\
 $1$ & $2$ & $1$    \\
$1$ & $3$ & $1$   \\
$1$ & $4$ & $2$   \\
\hline
\end{tabular}
\notag &
\begin{tabular}
     {|r|r|r|}
     \hline
$v_\mathbf{i}$ & $\Delta_\mathbf{i}$ & $\phi_\mathbf{i}$ \\    
\hline    
$2$ & $-4$ & $-1$   \\
$2$ & $-3$ & $-1$   \\
$2$ & $-2$ & $0$  \\
$2$ & $-1$ &  $1$ \\
$2$ & $0$ &  $1$ \\
$2$ & $1$ &  $2$   \\
$2$ & $2$ & $2$    \\
$2$ & $3$ & $3$   \\
$2$ & $4$ & $3$   \\
\hline
\end{tabular}
  \end{eqnarray}
    \caption{Function $\phi_\mathbf{i}(v_\mathbf{i},\Delta_\mathbf{i})$.}
    \label{table1}
\end{table}
For $\mathbf{i}\in [1;\mathbf{N}]$, controller $\mathbf{C}_{\mathbf{i}}$ takes as input, temperature of rooms $\mathbf{i}$, $\mathbf{i}-1$ and $\mathbf{i}+1$, and outputs control action $v_{\mathbf{i}}$ that is applied to $\mathbf{P}_{\mathbf{i}}$. 
Specification for room $1$ is given by the sequence $19, 21, 18, 17, 19, 20$, 
while specification for other rooms is to maintain temperature constant at $19$ Celsius degrees. 
We set $\theta_{\mathbf{i}}=1.5$, $\mathbf{i}\in [1;\mathbf{N}]$. Resulting NoS is given by $\mathcal{N}=(\mathcal{V},\mathcal{E},h)$, where:  
$\mathcal{V}=\{\Sigma_1,\Sigma_2,...,\Sigma_{\mathbf{N}},\Sigma_{\mathbf{N}+1},\Sigma_{\mathbf{N}+2},..., \Sigma_{N}\}$ with $\Sigma_i=\mathbf{P}_{\mathbf{i}}$, $\Sigma_{i+\mathbf{N}}=\mathbf{C}_{\mathbf{i}}$ for $i\in [1;\mathbf{N}]$, and $N=2\mathbf{N}$;
pair $(\Sigma_{\mathbf{i}},\Sigma_{\mathbf{j}})\in\mathcal{E}$ if 
$\mathbf{j}=\mathbf{i}+1$, for $\mathbf{i}\in  [1;\mathbf{N}-1]$, or
$\mathbf{j}=1$, for $\mathbf{i}=\mathbf{N}$, or 
$\mathbf{i}\in \{\mathbf{j}-\mathbf{N}, \mathbf{j}-\mathbf{N}-1,\mathbf{j}-\mathbf{N}+1\}$, for $\mathbf{j}\in [\mathbf{N}+1;N]$;
$h_{\mathbf{j},\mathbf{i}}:X_\mathbf{j}\rightarrow U_{\mathbf{j},\mathbf{i}}$ is defined for any $(\Sigma_\mathbf{j},\Sigma_\mathbf{i}) \in \mathcal{E}$  by $h_{\mathbf{j},\mathbf{i}}(x_\mathbf{j})=x_\mathbf{j}$, $x_\mathbf{j} \in X_\mathbf{j}$. By the symmetry of the geometric configuration of the rooms in the circular building and by the specifications considered, it is readily seen that instead of solving Problem \ref{problem1} for each $\Sigma_\mathbf{i}=\mathbf{P}_{\mathbf{i}}$, $\mathbf{i}\in [1;\mathbf{N}]$, we only need to consider three cases: $\mathbf{i}=1$, $\mathbf{i}\in [2;\mathbf{N}]$, 
$\mathbf{i}\in [3;\mathbf{N}-1]$.
We start by defining for each $\mathbf{i}\in [1;\mathbf{N}]$, the sequence of sub-NoS $\{\mathcal{N}(k)\}_{k\in\mathbb{N}}$ as in (\ref{NdiK}), resulting in: NoS $\mathcal{N}(0)=(\mathcal{V}_0,\mathcal{E}_0,h^0)$ is given by  $\mathcal{V}_0=\mathcal{V}
    $, $\mathcal{E}_0=\bigcup_{\mathbf{i}\in [1;\mathbf{N}]}\{(\Sigma_\mathbf{i},\Sigma_{\mathbf{i}+\mathbf{N}}),$ $(\Sigma_{\mathbf{i}+\mathbf{N}},\Sigma_\mathbf{i})\}$ and $h^0$ is derived accordingly;
    NoS $\mathcal{N}(1)=\mathcal{N}$. We start solving Problem \ref{problem1} by using NoS $\mathcal{N}(0)$. Consider room 2 and control sequence 
$v_2(.)=2,\,1,\,0,\,1,\,1$.
Such $v_2(.)$ allows enforcing desired temperature in room 2 up to accuracy $\theta_{2}$, irrespective of control sequences $v_{\mathbf{i}}(.)$ chosen for rooms $\mathbf{i}=1,3,4,...,\mathbf{N}$ as we now discuss. 
Dynamics of $x_2(.)$ is influenced only by $x_1(.)$ and $x_3(.)$.
By choosing $v_{\mathbf{i}}(t)=0$, $\mathbf{i}=1,3$, $t \in [0;5]$, resulting temperature is 
$x_2(.)=19,\,20,\,19,\,19,\,19,\,19$, hence specification for room 2 is met.  
By choosing $v_{\mathbf{i}}(t)=2$, $\mathbf{i}=1,3$, $t \in [0;5]$, we get 
$x_2(.)=19,\,20,\,20,\,20,\,20,\,20$, hence 
specification for room 2 is met.
By monotonicity of $x_2(.)$ with respect to $v_1(.)$ and $v_3(.)$, i.e. if $v_1(.)$ and $v_3(.)$ decrease/increase then $x_2(.)$ decreases/increases, for any choice of control sequences $v_1(.)$ and $v_2(.)$ (both lower bounded by $0$ and upper bounded by $2$),    
control sequence $v_2(.)$ allows enforcing desired temperature in room 2 up to accuracy $\theta_{2}$. 
On the other hand, given control sequence $v_2(.)$ for room 2, no control sequence $v_1(.)$ exists for enforcing desired specification on room 1. As a consequence, no solution is found for Problem \ref{problem1}, by using NoS $\mathcal{N}(0)$. 
For NoS $\mathcal{N}(1)$, sequences $v_\mathbf{i}(.)$ solving Problem \ref{problem1} are found and detailed in Table \ref{table2} that reports also the corresponding temperatures $x_{\mathbf{i}}$.
\begin{table}[]
    \centering
\begin{tabular}
     {|l||l|l||l |l || l | l ||}
\hline    
$t$ & $v_1$ & $x_1 $  & $v_2=$  & $x_2= $  & $v_{\mathbf{i}},$ & $x_{\mathbf{i}}$\\
& & & $v_{\mathbf{N}}$ & $x_{\mathbf{N}}$ & $\mathbf{i}\in [3;\mathbf{N}-1]$  & 
\\
\hline 
$0$ & $2$ & $19$  & $2$ & $19$ & $1$ & $19$\\
$1$ & $2$ & $20$  & $1$ & $19$ & $1$ & $19$\\
$2$ & $0$ & $19$  & $1$ & $20$ & $1$ & $19$\\
$3$ & $0$ & $18$  & $1$ & $20$ & $1$ & $19$\\
$4$ & $2$ & $18$  & $1$ & $20$ & $1$ & $19$\\
$5$ &  & $20$  &  & $19$ &  & $19$\\
\hline
\end{tabular}  
    \caption{Sequences $v_{\mathbf{i}}(.)$ and $x_{\mathbf{i}}(.)$, $\mathbf{i}\in [1;\mathbf{N}]$.}
    \label{table2}
\end{table}
We stress that while a decentralized control architecture,  corresponding to $\mathcal{N}(0)$, fails to find a solution, the architecture of NoS, corresponding to $\mathcal{N}(1)$,  finds a solution. It is readily seen that for $\mathbf{N}$ sufficiently large, a centralized approach fails to find a solution due to corresponding TCC burden and inherent bounded computational resources of any computing units. 
\section{Conclusions and Outlook}\label{sec9}
We introduced a novel control architecture, terms Network of Systems, which extends decentralized control architectures, and propose an incremental approach to design local controllers enforcing on each plant a regular language local specification, up to desired accuracies. 
Future work will focus on development of software tools implementing proposed algorithms and, on decentralized control architectures with an incremental design approach, as introduced in this paper.

\bigskip

\bibliographystyle{plain}
\bibliography{biblionew}

	\end{document}